\documentclass[12pt]{amsart}
\usepackage{amsmath,amssymb,amsthm}
\usepackage{dsfont}
\usepackage{mathrsfs}
\usepackage[margin=3cm]{geometry}
\usepackage{graphicx}\usepackage{subfigure}
\usepackage{multicol}
\usepackage{tikz}
\usepackage[all]{xy}
\usepackage{xfrac}
\usepackage{xcolor}
\definecolor{olive}{rgb}{0.3, 0.4, .1}
\definecolor{fore}{RGB}{249,242,215}
\definecolor{back}{RGB}{51,51,51}
\definecolor{title}{RGB}{255,0,90}
\definecolor{dgreen}{rgb}{0.,0.6,0.}
\definecolor{gold}{rgb}{1.,0.84,0.}
\definecolor{JungleGreen}{cmyk}{0.99,0,0.52,0}
\definecolor{BlueGreen}{cmyk}{0.85,0,0.33,0}
\definecolor{RawSienna}{cmyk}{0,0.72,1,0.45}
\definecolor{Magenta}{cmyk}{0,1,0,0}
\usepackage{url}
\usepackage{abstract}
\usepackage[T1]{fontenc}

\usepackage[bookmarksdepth=0]{hyperref}
\hypersetup{   
    colorlinks=true,
    citecolor=Magenta,
   linkcolor=blue,
    linktoc=all
                         }

\newtheorem{defn}{Definition}[section]
\newtheorem{thm}[defn]{Theorem}
\newtheorem{lem}[defn]{Lemma}
\newtheorem{prop}[defn]{Proposition}

\newtheorem{rmk}[defn]{Remark}

\newcommand{\R}{\mathds{R}}

\newcommand{\x}{\times}
\newcommand{\f}{\frac}
\newcommand{\D}{\mathcal{D}}
\newcommand{\K}{\mathds{k}}
\newcommand{\F}{\mathscr{F}}
\newcommand{\GG}{\mathscr{G}}
\newcommand{\PC}{\mathscr{P}_C}
\newcommand{\QC}{\mathscr{Q}_C}

\newcommand{\bu}{\bullet}
\newcommand{\s}{\mathscr{S}}

\newcommand{\EC}{\mathcal{E}_C}

\makeatletter
\@namedef{subjclassname@2020}{\textup{2020} Mathematics Subject Classification}
\makeatother

\title[Complete Projector]{Microlocal Projector for Complete Flow}

\author{Sheng-Fu Chiu}
\email{sfchiu@gate.sinica.edu.tw}
\address[Sheng-Fu Chiu]{Institute of Mathematics, Academia Sinica, Taiwan}

\subjclass[2020]{53E50, 57R17, 18G80}
\keywords{Microlocal projector, Hamiltonian dynamics, Derived Categories}

\begin{document} 

\maketitle

\begin{abstract}

In this note we generalize the construction of microlocal projector to the sublevel set of autonomous function with complete Hamiltonian flow under some mild conditions. Furthermore, we mention that the condition of being complete is optimal, by providing a non-example that violates the projector property.

\end{abstract}

\setcounter{tocdepth}{2}
\tableofcontents

\section{Introduction}

In symplectic geometry and topology one often considers Hamiltonian diffeomorphisms with compact supports on a symplectic manifold. There are natural reasons regarding such assumption. First, the flow of such Hamiltonian is always complete. Second, they are bounded deformations of the identity and thus distinguish symplectic diffeomorphisms from conformal symplectic diffeomorphisms. Third, the natural compact-open topology on the group of Hamiltonian diffeomorphisms is metrizable and has a countable basis. Nevertheless, relaxations of the condition of compact support are available under certain circumstances. 

Let $Q$ be a connected smooth manifold and denote the coordinates of its cotangent bundle $T^*Q$ by $(q,p)$. In \cite{Chiu17} the author constructed a microlocal projector and used it to prove a nonsqueezing result of contact balls. The existence of the microlocal ball projector relies on the geometric fact that on $T^*\R^n$ the Hamiltonian function $H_1(q,p)=q^2+p^2$  satisfies the following condition: 
\begin{center}
	For each $b\in\mathds{R}$, the set $H_1^{-1}((-\infty,b])$ is compact.
\end{center}
We also refer the reader to \cite[Definition 1.4]{KO21} where Kawasaki and Orita made the same assumption to the moment map on $T^*Q$ and studied its rigid fibers.

On the other hand, even the simplest linear motion in Hamiltonian dynamics has its Hamiltonian $H_2(q,p)=\frac{1}{2}p^2$, which does not satisfy the above condition. Instead it satisfies the following coercive condition:
\begin{center}
	For each compact set $K\subset Q$ and each $b\in\mathds{R}$,\\ the set $\{(q,p)|q\in K, H_2(q,p)\leq b\}$ is compact.
\end{center}
Likewisely in Lagrangian dynamics there are convex and superlinear conditions for the sake of the well-definition of the Legendre transformation.

One could also consider other Hamiltonians such as the superposition of two harmonic oscillators $H_3(q,p)=({q_1}^2+{p_1}^2)-2({q_2}^2+{p_2}^2)$. The function $H_3$ does not obey the above coercive condition, nor does it possess any compact level surface. That being said, it is easy to see that every single level surface of $H_3$ is foliated by compact invariant tori, and on each such leaf the Hamiltonian flow lives and lasts forever. 

A rather natural condition one could ask for is the following condition ($\star$):
\begin{defn}\label{star}
A time-independent Hamiltonian function $H:T^*Q\rightarrow \R$ satisfies the condition ($\star$) if:

\begin{enumerate}
	\item (Long-term Completness) The solution $h_s\in Ham(T^*Q)$ of the Hamiltonian equations $(\dot{q}(s),\dot{p}(s))= X_H(q(s),p(s))$ are defined for all $s\in\R$.\\
	\item (Short-term Separating) There is a connected open interval $J$, with $J=-J$, such that for any given $0<s\in J$ and $q_1,q_2\in Q$ there exists at most one Hamiltonian trajectory $\gamma$ from $T^*_{q_1}Q$ to $T^*_{q_2}Q$ with time $s$.
\end{enumerate}
\end{defn}

One sees that the above condition ($\star$) applies to many circumstances in symplectic geometry and dynamical systems. In this note we refine the argument in \cite{Chiu17} to construct microlocal projectors for those $H$ satisfying the condition ($\star$). Our result is:

\begin{thm}[See Theorem \ref{projector} for its definition of mircolocal projector]
 Suppose $H:T^*Q\rightarrow\mathds{R}$ is a function satisfying the condition ($\star$). Let $B_C=H^{-1}((-\infty,C))$ be an open sublevel set of level $C$. Then there exists a microlocal projector $\PC$ for $B_C$.
\end{thm}

A motivation for complete Hamiltonian flow comes from quantum mechanics. By Stone's theorem the quantum Hamiltonian observable $\hat{H}$ is always given self-adjoint, and its evolution on the quantum system is always governed by the one-parameter unitary group $e^{is\hat{H}}$, which is automatically defined for all times $s\in\R$. The microlocal projector in a derived category, in a loose sense, is a classical geometric analogue of the quantum projection operator that plays an essential role in quantum measurements.

We would like to refer the reader to \cite{MVZ12} that partial quasi-morphisms and quasi-states, obtained from spectral invariants, can be defined for Hamiltonians with complete flow. The authors of \cite{MVZ12} also suggested a possible generalization of Mather's alpha function for such Hamiltonians. We hope that in the future our construction of microlocal projector could shed some light on this matter.

\section{Theory of microlocal sheaves}

\subsection{Sheaves and their microlocal supports}

Let $\K$ be a ring. Let $D(Q)$ be the derived category of sheaves of $\K$-modules on $Q$. For any object $\F\in\D(Q)$, Kashiwara and Schapira \cite{KS90} define its \emph{microsupport} (denoted by $SS(\F))$ to be the closure of those $(q_0,p_0)\in T^*Q$ such that there exists a $C^1$ function $\phi:Q\rightarrow\mathbb{R}$ with $d\phi(q_0)=p_0$ and satisfying the following ill propagation condition:
\begin{equation}
(R\Gamma_{\{q\in Q|\phi(q)\geq\phi(q_0)\}}\F)_{q_0}\ncong0 .
\end{equation}

Functorial relations between microsupports are extensively used in the determination of subcategories of certain sheaves. Their explanations and proofs can be found in \cite{GKS12,KS90,Tamarkin18}. For the convenience of the reader, we list several of them which take roles in the present paper.

Let $Y\xrightarrow{f} X$ be a morphism of smooth manifolds. It induces the natural morphisms 
$$
T^*X\xleftarrow{f_\pi} T^*X\x_X Y\xrightarrow{f^t} T^*Y.
$$

The microsupports of the images under Grothendieck's six functors($Rf_*, Cf_!, f^{-1}$, $f^!, \otimes, \underline{Rhom}$) are bounded by the microsupports of their preimages, via $f^t$ and $f_\pi$, in the sense of the following propositions:

\begin{prop}\label{push} \cite[Proposition 5.4.4]{KS90}
	Let $\GG\in D(Y)$ and assume $f$ is proper on $supp(\GG)$, then 
	$$
	SS(Rf_*\GG)=SS(Rf_!\GG)\subset f_\pi((f^t)^{-1}(SS(\GG))).
	$$
\end{prop}

\begin{prop}\label{pull} \cite[Proposition 5.4.13]{KS90}
	Let $\F\in D(X)$ and assume $f$ is non-characteristic for $SS(\F)$, i.e., $f_\pi^{-1} (SS(\F))\cap T_Y^*X\subset Y\x_X T_X^*X$. Then	
	\begin{enumerate}	
	\item $SS(f^{-1}\F)\subset(f^t)(f_\pi^{-1}(SS(\F)))$.
	
	\item $ f^{-1}\F \otimes \omega_{Y/X} \cong f^!\F$.
	\end{enumerate}
\end{prop}

\begin{prop}\label{tenhom} \cite[Proposition 5.4.14]{KS90} Let $\F,\GG\in D(X)$. Then
   \begin{enumerate}
	\item $    SS(\F)\cap SS(\GG)^a\subset T_X^*X\Rightarrow SS(\F\otimes\GG)\subset SS(\F)+SS(\GG)$.    
	\item $  SS(\F)\cap SS(\GG)\subset T_X^*X\Rightarrow SS(\underline{Rhom}(\F;\GG))\subset SS(\F)^a + SS(\GG)$.
   \end{enumerate}	
	
\end{prop}

When it comes to the products and projections, a sharper control of microsupports is available. Let $E$ be a nontrivial finite-dimensional real vector space and $\pi:X\x E\rightarrow X$ be the projection map.  

\begin{prop}\label{projection} \cite[Corollary 3.4]{Tamarkin18}
	Let $\kappa:T^*X\x E\x E^\star\rightarrow T^*X\x E^\star$ be the projection and let $i:T^*X\rightarrow T^*X\x E^\star$ be the closed embedding $(x,\omega)\mapsto(x,\omega,0)$ of zero-section. Then for $\F\in D(X\x E)$ we have
	$$
	SS(R\pi_!\F),SS(R\pi_*\F)\subset i^{-1}\overline{\kappa(SS(\F))}.
	$$
\end{prop}

When $E$ is replaced by a smooth manifold $Y$, the same argument shows the following: If an object $\F\in D(X\x Y)$ is nonsingular in a neighborhood of $\{(x,\omega)\}\x T_Y^*Y$ in $T^*(X \x Y)$, then its projection along $Y$ will remain nonsingular at $(x,\omega)$.

\subsection{The Tamarkin categories $\D$} It was shown in \cite{KS90} that for any $\F$ in $D(Q)$ the correpsonding microsupport $SS(\F)$ is always a closed $\R^+$-conic coisotropic subset of $T^*Q$. To tame the conic structure, Tamarkin \cite{Tamarkin18} considered certain subcategories of the category $D(X\x\R)$ as follows. Denote the coordinate of $T^*(Q\x\R)$ by $(q,p,z,\zeta)$ where $(q,p)\in T^*Q$, $z\in\R$, and $\zeta\in T^*_z \R$. Following \cite{Tamarkin18}, we consider the following full subcategory of $D(Q\x\R)$:
$$
 D_{\zeta\leq0}=\{\mathcal{F}\in D(Q\x\R) | SS(\F)\subset \{\zeta\leq0\}\}.
$$

\begin{defn}
The Tamarkin category is defined to be
\begin{equation}
\D(Q\x\R):= D_{\zeta\leq0}(Q\times\R)^{\perp,l}
\end{equation}
Here the left orthogonal complement is taken with respect to $Rhom$ in $D(Q\times\R)$.
\end{defn}

A remarkable feature of about Tamarkin category $\D(Q\x\R)$ is that there exists a quotient functor $D(Q\x\R)\rightarrow \D(Q\x\R)$. Moreover, this quotient functor can be represented by a kernel convolution.

Let $\F\in D(Q_1\x Q_2 \x\R)$ and $\GG \in D(Q_2\x Q_3 \x\R)$. We define their convolution to be
\begin{equation}
\F\bu_{Q_2}\GG= R\pi_{13\sigma !}({\pi_{Q12R1}}^{-1}\F\otimes{\pi_{Q23R2}}^{-1}\GG)\in D(Q_1\x Q_3\x\R)
\end{equation}
where $\pi_{Q12R1}:Q_1\x Q_2\x Q_3\x\R_1\x\R_2\rightarrow Q_1\x Q_2\x\R_1$ and $\pi_{Q23R2}:Q_1\x Q_2\x Q_3\x\R_1\x\R_2\rightarrow Q_2\x Q_3\x\R_2$ are projections and $\pi_{13\sigma}:Q_1\x Q_2\x Q_3\x\R_1\x\R_2\rightarrow Q_1\x Q_3\x \R$ maps $(q_1,q_2,q_3,z_1,z_2)$ to $(q_1,q_3,z_1+z_2)$.

When $Q_2$ and $Q_3$ are single points, by abuse of notation we denote by $\F\bu\GG$ the convolution product $\F\bu_{Q_2}\GG$ for $\F\in D(Q\x\R)$ and $\GG\in D(\R)$. Therefore convoluting with a given object of $D(\R)$ becomes a functor $D(Q\x\R)\rightarrow D(Q\x\R)$.

For any subset $A\subset \R$, let $\K_A\in D(\R)$ be the constant sheaf supported by $A$. We also denote $\K_0=\K_{\{z=0\}}$, $\K_{\geq c}=\K_{[c,\infty)}$ and $\K_{>c}=\K_{(c,\infty)}$. For any $\F\in D(Q\x\R)$ we have $\F\bu\K_0=\F$. Moreover:

\begin{prop}\cite[Proposition 2.1, Proposition 2.2]{Tamarkin18}\label{positive}
Convolution with the exact triangle
\begin{equation}
\K_{\geq0}\rightarrow\K_0\rightarrow\K_{>0}[1]\xrightarrow{+1}
\end{equation}
gives rise to the semiorthogonal decomposition with respect to the triple $(\D(Q\x\R),D(Q\x\R),D_{\zeta\leq0}(Q\times\R))$.
\end{prop}

Another important fact about the Tamarkin category theory is that the category $\D(Q\x\R)$, originally defined to be the left orthogonal complement of  $D_{\zeta\leq0}(Q\times\R)$, is also the right orthogonal complement of the same category $D_{\zeta\leq0}(Q\times\R)$. The quotient functor from $D(Q\x\R)$ to $ D_{\zeta\leq0}(Q\times\R)^{\perp,r}$ is represented by $\hom^{\star}(\K_{\geq0},-)$, the adjoint functor of the convolution with $\K_{\geq0}$. 

To summarize, the situation of the Tamarkin category $\D(Q\x\R)$ inside $D(Q\x\R)$ is illustrated by the following diagram:
\begin{center}
\begin{equation}\label{D}
\begin{tabular}{c}

\xymatrix
{ 
         \ar @{} [drr] |{\bigcup}  
	& D \ar@/_/[dl]_{(-)\bu\K_{\geq0}}  \ar@/^/[dr]^{\hom^{\star}(\K_{\geq0},-)} &  \\
	{D_{\zeta\leq0}}^{\perp,l}  &  \D \ar@2{-}[l] \ar@2{-}[r] & {D_{\zeta\leq0}}^{\perp,r}
}

\end{tabular}
\end{equation}
\end{center}

Now we introduce the notion of Tamarkin subsest categories:
\begin{defn}[Tamarkin Subset Category]
	Consider the de-homogenization map $\rho:T^*Q\x\dot{T}^*\R\rightarrow T^*Q$, $\rho(q,p,z,\zeta)=(q,\frac{p}{\zeta})$. 
	\begin{enumerate}
		\item For $A \overset{cls}\subset T^*Q$, we define $\D_A(Q\x\R):=\{\F\in \D| SS(\F)\subset \rho^{-1}(A)\}$.
		\item	For $U\overset{open}\subset T^*Q$, we define $\D_U(Q\x\R):={D_{T^*Q\setminus U}}^{\perp,l}$ w.r.t $Rhom$ in $\D$.
	\end{enumerate}
\end{defn}

\subsection{Sheaf quantization}

Let $I\subset\R$ be an open interval (connected subset of $\R$ with non-empty interior) containing the origin and let
$h:I\x T^*Q\rightarrow T^*Q$ be the Hamiltonian diffeomorphism generated by an autonomous Hamiltonian function $H: T^*Q\rightarrow\R$.

It is convenient to construct such sheaf quantization using the actional functional (which is implicitly carried out in Guillermou-Kashiwara-Schapira's \cite[Proposition A.6]{GKS12}). Define the action functional $S:I\x T^*Q\rightarrow\R$ by $S(s_1;q_1,p_1)=\int_\gamma pdq-Hds$ where $\gamma$ is the Hamiltonian trajectory $\gamma(s)=h_s(q_1,p_1)$ connecting $\gamma(0)=(q_1,p_1)$ and $\gamma(s_1)=(q_2,p_2)$. The Hamiltonian diffeomorphism $h$ lifts to a  Hamiltonian diffeomorphism $\tilde{h}:I\x T^*Q\x T^*_{>0}\R\rightarrow T^*Q\x T^*_{>0}\R$ by 
\begin{equation}\label{liftto}
\tilde{h}_s(q_1,p_1;z,\zeta)=(q_2,p_2;z+S(s;q_1,\frac{p_1}{\zeta}),\zeta),
\end{equation}
where $(q_2,p_2)=h_s(q_1,p_1)$. The map $\tilde{h}$ is homogeneous in the sense that $h\circ (id_I\x\rho )=\rho\circ\tilde{h}$ as maps from $I\x T^*Q\x T^*_{>0}\R$ to $T^*Q$. In other words, the following diagram commutes:
\begin{center}
	\begin{equation}
	\begin{tabular}{c}
		\xymatrix{
			{{T_{\zeta>0}^*}(Q\x\R)\x I}\ar[d]_{\rho\x id_I} \ar[rr]^{\tilde{h}}  &   & {{T_{\zeta>0}^*}(Q\x\R)}\ar[d]_{\rho}\\
			{T^*Q \x I}\ar[rr]^{h}&  & {T^*Q}.
		}
	\end{tabular}	
	\end{equation}
\end{center}

Let $\Lambda\subset T^*(I\x Q\x Q)\x\R$ be the Legendrian suspension of the graph of $h$, defined by
$$
\Lambda:=\{(-Hds;q_1,-p_1;q_2,p_2;-S(s;q_1,p_1))|(q_2,p_2)=h_s(q_1,p_1)    \}.
$$

Its homogenization $\tilde{\Lambda}\subset T^*(I\x Q\x Q\x\R)$ is defined by
\begin{equation}
\tilde{\Lambda}=\{(\zeta x;z,\zeta)|(x,z)\in\Lambda,\zeta>0\}, 
\end{equation}
which is just the Lagrangian suspension of the graph of $\tilde{h}$ in the sense of convolution product.

\begin{prop}\label{quantization}
Suppose that the flow of $h$ satisfies the condition ($\star$) and we take $I=\R$. Then there exists an object $\s\in\D(I\x Q\x Q\x\R)$ such that $\s|_{0\x Q\x Q\x\R}$ is isomorphic to $\K_{\Delta}$ in $\D$, where $\Delta=\{(q,q,z)|z\geq0\}$ and $SS(\s)\cap \{\zeta>0\} \subset \tilde{\Lambda}$. The object $\s$ is called the sheaf quantization of the Hamiltonian diffeomorphism $h$.
\end{prop}

\begin{proof}
Recall from Definition \ref{star} there is a connected open interval $J$ such that for any given $0<s\in J$ and $q_1,q_2\in Q$ there exists at most one Hamiltonian trajectory $\gamma$ from $T^*_{q_1}Q$ to $T^*_{q_2}Q$ with time $s$. For such $s\in J$ and such $q_1,q_2$ we can write $p_1=p_1(q_1,q_2)$ and rewrite the action functional $S(s;q_1,p_1)$ by $S(s;q_1,q_2)$. The local sheaf quantization is defined by
\begin{equation}
\s_J := \K_{\{(s,q_1,q_2,z) | z+S(s;q_1,q_2)\geq0 \} } \in \D(J\x Q\x Q\x\R).
\end{equation}
We then adapt the strategy from \cite[Lemma 3.5(B)]{GKS12} to glue to the global sheaf $\s\in\D(I\x Q\x Q\x\R)$.

\end{proof}

Let us point out that our notion of sheaf quantization is slightly different from that originally established in Guillermou-Kashiwara-Schapira's \cite{GKS12}. First, in \cite{GKS12} the sheaf quantization is defined for homogeneous Hamiltonian on $\dot{T}^*M$ with its zero section cut off. In the scenario of homogenizing $h$ on $T^*Q$ to $\tilde{h}$ on $\dot{T}^*M$ where $M=Q\x\R$, it is required by \cite{GKS12} that the pre-homogenized Hamiltonian $h$ satisfies compact support condition for (\ref{liftto}) to be defined at $\zeta=0$. Second, in \cite{GKS12} it is shown that the sheaf quantization $\s$ satisfies $\s|_{s=0}\cong \K_\Delta$ in the ambient category $D(I\x M\x M)=D(I\x Q\x Q \x\R\x\R)$. In constrast to \cite{GKS12} we only require the sheaf quantization at time $s=0$ to be isomorphic to $\K_\Delta$ within the Tamarkin category $\D(I\x Q\x Q\x\R)$, and instead of the composition product we use the convolution product which is more compatible with $\D$. This slightly weaker isomorphism allows us to treat certain $h$ not necessarily extendable to $\dot{T}^*(Q\x\R)$ and drop the compact support condition.

For example let us take $H=\f{1}{2}p^2$ then one has $h_s(q_1,p_1)=(q_1+sp_1,p_1)$ and $S(s;q_1,p_1)=\f{-1}{2}s{p_1}^2 $. Therefore $S(s;q_1,q_2)=\f{-1}{2s}(q_2-q_1)^2$ for $s\neq0$. Let $\s := \K_{\{ (s,q_1,q_2,z)| z+ S(s;q_1,q_2)\geq0  \}}$ be the associated sheaf quantization. We see that when $q_1\neq q_2$, let $s\rightarrow 0^+$ then $z$ has no support at all. On the other hand let $s\rightarrow 0^-$ then we see that $z$ is supported on the whole $\R$. These lead to $\s|_{s=0}\cong \K_\Delta$ in the Tamarkin category $\D$ rather than in the ambient category $D$.

\section{Microlocal Projector}

\subsection{Projectors associated to sublevel sets} This subsection is devoted to the construction of the projection kernel associated with the open sublevel set. We assume
$$
H : T^*Q\rightarrow\R
$$
is an autonomous Hamiltonian function satisfying the condition ($\star$) (See Definition \ref{star}).

For this Hamiltonian $H$ we take the object $\s\in \D(I\x Q\x Q\x\R)$ obtained in Proposition \ref{quantization}. By the completeness of $H$ we choose the time interval $I$ to be the whole $\R$. We stick to the notation $I$ in order to distinguish the time (in $s$) with the extra action variable $\R$ (in $z$). We denote the dual of $I$ by $I^*$ (in $t$).

Let us perform the Fourier-Sato-Tamarkin transform on $\s$ with respect to the time variable $I\rightsquigarrow I^*$ :
\begin{equation}
\widehat{\s}:= \s\underset{s}{\bu}\K_{\{(s,t,z)|z+st\geq0\}}[1]\in \D(I^*\x Q\x Q\x\R)
\end{equation}

In the realm of functional analysis, the Fourier transform is expected to intertwine uniform distributions and delta functions. Similarly in algebraic microlocal analysis we have the following property for the Fourier-Sato-Tamarkin transform:
\begin{equation}
\widehat{\s}\underset{t}{\circ}{\K_{\R}}=\s\underset{s}{\bu}\K_{\{(s,t,z)|z+st\geq0\}}[1]\underset{t}{\circ}{\K_{\R}}\cong \s\underset{s}{\bu}\K_{\{s=0,z\geq0\}}=\s|_{s=0} 
\end{equation}
which is, by Theorem \ref{quantization},
\begin{equation}
\widehat{\s}\underset{t}{\circ}{\K_{\R}}\cong \K_{\Delta}.
\end{equation}

Note that in the above formulation we write the composition product  "$\underset{t}{\circ}$" instead of the convolution product "$\underset{t}{\bu}$" under the convention that $\K_{\R}=\K_{\R_t}$ lacks the $z$-component. It is easy to verify that the composition $(-)\underset{t}{\circ}{\K_{\R}}$ is equivalent to the convolution $(-)\underset{t}{\bu}{\K_{\{(t,z)|t\in\R,z\geq0\}}}$ in $\D$.

Let $C$ be a real number and let $B_C=\{(q,p)\in T^*Q | H(q,p)<C \}$ be the open sublevel set of $H$ of level $C$.

\begin{defn}[Microlocal Projectors]
We define the microlocal projectors associated to $B_C$ and $T^*Q\setminus B_C$ respectively by
\begin{equation}
\PC:=\widehat{\s}\underset{t}{\circ}\K_{\{t<C\}}
\end{equation} 
and 
\begin{equation}
\QC:=\widehat{\s}\underset{t}{\circ}\K_{\{t\geq C\}}.
\end{equation}
respectively.
\end{defn}

Together with $\K_{\Delta}$ they form the following exact triangle
\begin{equation}
\PC\rightarrow\K_{\Delta}\rightarrow \QC \xrightarrow{+1}
\end{equation}
which can be obtained by composing $\widehat{\s}$ with the fundamental triangle 
\begin{equation}
\K_{\{t<C\}} \rightarrow\K_{\R}\rightarrow \K_{\{t\geq C\}} \xrightarrow{+1}.
\end{equation}

Since $\PC$, $\QC$, and $\K_{\Delta}$ are objects of $\D(Q\x Q\x\R)$, convolutions with them give rise to functors from $\D(Q\x\R)$ to $\D(Q\x\R)$. Moreover we have the following theorem:

\begin{thm}\label{projector}
The convolution with the exact triangle
\begin{equation}
\PC\rightarrow\K_{\Delta}\rightarrow \QC \xrightarrow{+1}
\end{equation}
gives rise to the semiorthogonal decomposition with respect to the triple of subset categories
\begin{equation}
(\D_{B_C}(Q\x\R),\D(Q\x\R),\D_{T^*Q\setminus B_C}(Q\x\R)).
\end{equation}
In particular, we have idempotent functors $\PC\bu\PC\cong\PC$ and $\QC\bu\QC\cong\QC$ (this is where the term projector comes from).

\end{thm}

\subsection{Proof of Theorem \ref{projector}}

Before proving the Theorem, let us briefly explain how to deduce the idempotency of the projector $\PC$ (similarly for $\QC$) from the desired  semiorthogonal decomposition. For any $\F\in\D(Q\x\R)$ we have its semidecomposition represented by the following exact triangle
\begin{equation}
\PC\bu\F\rightarrow\F\rightarrow \QC\bu\F \xrightarrow{+1}.
\end{equation}
We then perform the operation $\PC\bu(-)$ and get the following exact triangle
\begin{equation}
\PC\bu\PC\bu\F\rightarrow\PC\bu\F\rightarrow \PC\bu\QC\bu\F \xrightarrow{+1}.
\end{equation}
On the other hand, by the semiorthogonal decomposition of the object $\PC\bu\F$ there is an exact triangle
\begin{equation}
\PC\bu\PC\bu\F\rightarrow\PC\bu\F\rightarrow \QC\bu\PC\bu\F \xrightarrow{+1}.
\end{equation}
It is easy to see that the first morphisms in both triangles are identical. By the uniqueness of the cone one has the equality $\PC\bu\QC\bu\F\cong\QC\bu\PC\bu\F$. This object belongs to both the category $\D_{T^*Q\setminus B_C}(Q\x\R)$ and its left orthogonal complement $\D_{T^*Q\setminus B_C}(Q\x\R)^{\perp,l}$, hence $\PC\bu\QC\bu\F\cong\QC\bu\PC\bu\F\cong 0$ is isomorphic to the zero object of $\D(Q\x\R)$. So we have the isomorphism $\PC\bu\PC\bu\F\cong\PC\bu\F$ functorially for any $\F\in\D(Q\x\R)$ and we deduce the idempotency of $\PC$.

Now we turn to the main body of the proof. First we verify that the convolution with $\QC$ projects to the subcategory $\D_{T^*Q\setminus B_C}(Q\x\R)$. 
\begin{lem}\label{right}
	For any object $\GG\in\D(Q\x\R)$, one has $\GG\underset{Q}{\bu}\QC\in \D_{T^*Q\setminus B_C}(Q\x\R)$.
\end{lem}

\begin{proof}[Proof of Lemma \ref{right}]
The proof is word for word valid for the \textit{Step 1} of the proof of \cite[Theorem 3.11]{Chiu17}. For the reader's convenience we reproduce the argument \textit{loc. cit.} as follows:

Denote by $Q_1=Q_2=Q$ the base manifold and denote by $\R_1=\R_2=\R$ the extra action variable. We take $\GG\in\D(Q_1\x\R_1)$ and $\widehat{\s}\in\D(I^*\x Q_1\x Q_2\x\R_2)$. Let $\pi_{Q1R1}:I^*\x Q_1\x Q_2\x\R_1\x\R_2\rightarrow Q_1\x\R_1$ and $\pi_{IQ12R2}:I^*\x Q_1\x Q_2\x\R_1\x\R_2\rightarrow Q_1\x Q_2\x\R_2$ be projection maps. Let $\pi_\sigma:I^*\x Q_1\x Q_2\x\R_1\x\R_2\rightarrow Q_2\x\R$ be the product of maps that take sum $\R_1\x\R_2\rightarrow\R$ and project $I^*\x Q_1\x Q_2\x\R_1\x\R_2$ onto $Q_2$.

By the definition of microlocal projector we have
\begin{equation}\label{adam}
\GG\underset{Q}{\bu}\QC=\GG\underset{Q}{\bu}\widehat{\s}\underset{t}{\circ}\K_{\{t\geq C\}}
= R{\pi_\sigma}_!({\pi_{Q1R1}}^{-1}\GG\otimes {\pi_{IQ12R2}}^{-1}\widehat{\s} \otimes \K_{[R,\infty)\x Q_1\x Q_2\x\R_1\x\R_2})
\end{equation}

In the cotangent bundle $T^*(I^*\x Q_1\x Q_2\x\R_1\x\R_2)$ one obtains the following estimates using Proposition \ref{pull} and Proposition \ref{quantization}:
\begin{equation}\label{akira}
\begin{split}
SS({\pi_{Q1R1}}^{-1}\GG)\subset \{(t,0;q_1,\zeta_1 p_1;q_2,0;z_1,\zeta_1;z_2,0)|\zeta_1\geq0 \}\\
SS( {\pi_{IQ12R2}}^{-1}\widehat{\s})\subset \{(t,\zeta_2 s;q_1,\zeta_2 p_1';q_2,\zeta_2 p_2';z_1,0;z_2,\zeta_2)|\zeta_2\geq0,\\
 \quad\quad H(q_1,p_1')=H(q_2,p_2')=t, (q_2,p_2')=h_s(q_1,p_1') \}.
 \end{split}
\end{equation}
From (\ref{akira}) it is clear that 
$$
SS({\pi_{Q1R1}}^{-1}\GG)\cap SS( {\pi_{IQ12R2}}^{-1}\widehat{\s})^a\subset T^*_{I^*\x Q_1\x Q_2\x\R_1\x\R_2}(I^*\x Q_1\x Q_2\x\R_1\x\R_2)
$$
Therefore by Proposition \ref{tenhom} we have
\begin{equation}\label{antonio}
\begin{split}
SS({\pi_{Q1R1}}^{-1}\GG\otimes {\pi_{IQ12R2}}^{-1}\widehat{\s})\subset SS({\pi_{Q1R1}}^{-1}\GG)+SS({\pi_{IQ12R2}}^{-1}\widehat{\s})\\
\subset\{(t,\zeta_2 s;q_1,\zeta_1 p_1 + \zeta_2 p_1';q_2,\zeta_2 p_2';z_1,\zeta_1;z_2,\zeta_2)    | \zeta_1\geq 0,\zeta_2\geq 0, H(q_1,p_1')=H(q_2,p_2')=t \}
\end{split}
\end{equation}

On the other hand, we know that
\begin{equation}\label{alfonso}
\begin{split}
SS(\K_{[R,\infty)\x Q_1\x Q_2\x\R_1\x\R_2}) &
=\{ (C,\zeta)|\zeta\geq0 \}\x T^*_{Q_1\x Q_2\x\R_1\x\R_2}(Q_1\x Q_2\x\R_1\x\R_2)\\
&\cup\{(t,0)|t>C \}\x T^*_{Q_1\x Q_2\x\R_1\x\R_2}(Q_1\x Q_2\x\R_1\x\R_2).
\end{split}
\end{equation}

Thus from (\ref{antonio}) and (\ref{alfonso}) we have

\begin{equation}
\begin{split}
SS({\pi_{Q1R1}}^{-1}\GG\otimes {\pi_{IQ12R2}}^{-1}\widehat{\s})\cap SS(\K_{[R,\infty)\x Q_1\x Q_2\x\R_1\x\R_2})^a\\
 \subset  T^*_{I^*\x Q_1\x Q_2\x\R_1\x\R_2}(I^*\x Q_1\x Q_2\x\R_1\x\R_2)
 \end{split}
\end{equation}

Apply Proposition \ref{tenhom} again we obtain
\begin{equation}\label{alex}
\begin{split}
SS({\pi_{Q1R1}}^{-1}\GG \otimes {\pi_{IQ12R2}}^{-1}\widehat{\s} \otimes \K_{[R,\infty)\x Q_1\x Q_2\x\R_1\x\R_2})\\
\subset SS({\pi_{Q1R1}}^{-1}\GG\otimes {\pi_{IQ12R2}}^{-1}\widehat{\s})+SS(\K_{[R,\infty)\x Q_1\x Q_2\x\R_1\x\R_2})\\
\subset\{(C,\zeta_2 s+\zeta;q_1,\zeta_1 p_1 + \zeta_2 p_1';q_2,\zeta_2 p_2';z_1,\zeta_1;z_2,\zeta_2)\\    
| \zeta_1\geq 0,\zeta_2\geq 0,\zeta\geq0, H(q_1,p_1')=H(q_2,p_2')=C \} \\
\bigcup \{(t,\zeta_2 s;q_1,\zeta_1 p_1 + \zeta_2 p_1';q_2,\zeta_2 p_2';z_1,\zeta_1;z_2,\zeta_2)\\
    | \zeta_1\geq 0,\zeta_2\geq 0, H(q_1,p_1')=H(q_2,p_2')=t>C \}.
 \end{split}
\end{equation}

We factor the map $\pi_\sigma:I^*\x Q_1\x Q_2\x\R_1\x\R_2\rightarrow Q_2\x\R$ into the composition of $g:(t,q_1,q_2,z_1,z_2)\mapsto(t,q_1,q_2,z_1,z_1+z_2)$ and $\pi_{Q2R}:(t,q_1,q_2,z_1,z)\mapsto (q_2,z)$. Then by applying Proposition \ref{push} on (\ref{alex}) we get 

\begin{equation}\label{albert}
\begin{split}
SS(Rg_!({\pi_{Q1R1}}^{-1}\GG \otimes {\pi_{IQ12R2}}^{-1}\widehat{\s} \otimes \K_{[R,\infty)\x Q_1\x Q_2\x\R_1\x\R_2}))\\
\subset\{(C,\zeta_2 s+\zeta;q_1,\zeta_1 p_1 + \zeta_2 p_1';q_2,\zeta_2 p_2';z_1,\zeta_1-\zeta_2;z_2,\zeta_2)\\    
| \zeta_1\geq 0,\zeta_2\geq 0,\zeta\geq0, H(q_1,p_1')=H(q_2,p_2')=C \} \\
\bigcup \{(t,\zeta_2 s;q_1,\zeta_1 p_1 + \zeta_2 p_1';q_2,\zeta_2 p_2';z_1,\zeta_1-\zeta_2;z_2,\zeta_2)\\
| \zeta_1\geq 0,\zeta_2\geq 0, H(q_1,p_1')=H(q_2,p_2')=t>C \}.
\end{split}
\end{equation}

Recall from (\ref{adam}) we have 
\begin{equation}
\GG\underset{Q}{\bu}\QC= R{\pi_{Q2R}}_!Rg_!({\pi_{Q1R1}}^{-1}\GG\otimes {\pi_{IQ12R2}}^{-1}\widehat{\s} \otimes \K_{[R,\infty)\x Q_1\x Q_2\x\R_1\x\R_2}).
\end{equation}

Suppose $(q_2,p,z_2,1)\in SS(\GG\underset{Q}{\bu}\QC)$ is a vector with nonzero component in $\zeta$. By applying Proposition \ref{projection} to the map $\pi_{Q2R}:(t,q_1,q_2,z_1,z)\mapsto (q_2,z)$ and the estimate ($\ref{albert}$) we deduce that there are sequences $\{\zeta_2\}$ and $\{p_2'\}$ subject to the constraint of (\ref{albert}) that satisfy $\{\zeta_2\}\rightarrow 1$ and $\{\zeta_2 p_2'\}\rightarrow p$. Therefore $\{p_2'\}\rightarrow p$ and we obtain $H(q_2,p)= \lim H(q_2,p_2')\geq C$.

\end{proof}

Next, we need to prove that convolution with $\PC$ projects to the left semiorthogonal complement of $\D_{T^*Q\setminus B}(Q\x\R)$.

\begin{lem}\label{left}
	For any object $\F$ of $\D_{T^*Q\setminus B_C}(Q\x\R)$ and object $\GG$ of $\D(Q\x\R)$, we have $Rhom(\GG \underset{Q}{\bu}\PC,\F)\cong 0$.
\end{lem}

\begin{proof}[Proof of Lemma \ref{left}]
 Again let $Q_1=Q_2=Q$ be our base manifold and let $\R_1=\R_2=\R$ be the extra action variable. Let 
$$
{I^*}\x Q_1\x Q_2\x\R_1\x\R_2\xrightarrow{\pi_{IQ1R1}}{I^*}\x Q_1\x\R_1\xrightarrow{\pi_{Q1R1}}Q_1\x\R_1
$$
be successive projections onto the domain of $\GG$, and 
$$
\pi_{\sigma}: {I^*}\x Q_1\x Q_2\x\R_1\x\R_2\xrightarrow{\sigma}{I^*}\x Q_2\x\R\xrightarrow{\pi_{Q2R}}Q_2\x\R
$$
onto the domain of $\F$. Here the map $\sigma$ performs summation $\R_1\x\R_2\rightarrow\R$ and projects the other components to a point. The map $\pi_{Q2R}$ is the projection along ${I^*}$.

In order to describe the restriction to $(-\infty,C)$, let us define  
\begin{equation}
\begin{split}
(-\infty,C)\x Q_1\x Q_2\x\R_1\x\R_2 & \xrightarrow{j} {I^*}\x Q_1\x Q_2\x\R_1\x\R_2  \\
(-\infty,C)\x Q_1\x\R_1 & \xrightarrow{j_1} {I^*}\x Q_1\x\R_1\\
(-\infty,C)\x Q_2\x\R & \xrightarrow{j_2} {I^*}\x Q_2\x\R
\end{split}
\end{equation}
to be the corresponding open embeddings of $(-\infty,C)$ into ${I^*}$. Define the projection map ${I^*}\x Q_1\x Q_2\x\R_1\x\R_2\xrightarrow{\pi_{IQ12R2}}{I^*}\x Q_1\x Q_2\x\R_2$ onto the domain of $\widehat{\s}$. \\
\begin{equation}
\resizebox{\displaywidth}{!}{%
\xymatrix{
	&{I^*}\x Q_1\x Q_2\x\R_2 & {I^*}\x Q_1\x Q_2\x\R_1\x\R_2  \ar@/_/[lldd]_{\pi_{IQ1R1}} \ar@/^3pc/[rdd]^{\sigma} \ar@/^8pc/[rddd]^{\pi_{\sigma}}   \ar[l]_{\pi_{IQ12R2}}   & \\
	& & (-\infty,C)\x Q_1\x Q_2\x\R_1\x\R_2 \ar@{^{(}->}[u]^{j}  \ar@{-->}[d]^{\sigma} \ar@{-->}[dl]^{\pi_{IQ1R1}}&   \\
	{I^*}\x Q_1\x\R_1\ar[d]_{\pi_{Q1R1}}   & (-\infty,C)\x Q_1\x\R_1\ar@{_{(}->}[l]_{j_1} &(-\infty,C)\x Q_2\x\R\ar@{^{(}->}[r]^{j_2}  &    {I^*}\x Q_2\x\R  \ar[d]_{\pi_{Q2R}} \\
	Q_1\x \R_1 & & &  Q_2\x\R \\
	&
}
}
\end{equation}

Recall that $\F\in\D(Q_2\x\R)$ and $\widehat{\s}\in\D(I^*\x Q_1\x Q_2\x\R_2)$. We define a functor $\EC:\D(Q\x\R)\rightarrow\D(Q\x\R)$  that sends any sheaves $\F$ to
\begin{equation}
\EC(\F):= R{\pi_{Q1R1}}_* R{j_1}_* R{\pi_{IQ1R1}}_* j^! \underline{Rhom}({\pi_{IQ12R2}}^{-1}(\widehat{\s}),{\pi_{\sigma}}^!\F)
\end{equation}

\begin{lem}
The functor $\EC$ is the right adjoint functor of $(-)\bu\PC$.
\end{lem}
\begin{proof}
	
Let us unwrap $\PC$ all the way using the maps in the above diagram and apply the adjunction relations between their corresponding Grothendieck six functors. Let $\GG$ be an object of $\D(Q_1\x\R_1)$. One has by definition, 
\begin{equation}
\GG\bu\PC=\GG\underset{Q}{\bu}(\widehat{\s}\underset{t}{\circ}\K_{\{t< R\}})=(\GG\underset{Q}{\bu}\widehat{\s})\underset{t}{\circ}\K_{\{t< R\}}=R{\pi_{Q2R}}_!R{j_2}_!{j_2}^{-1}(\GG\underset{Q}{\bu}\widehat{\s}).
\end{equation}
By unwrapping $\GG\underset{Q}{\bu}\widehat{\s}$ one obtain
\begin{equation}\label{aston}
\GG\bu\PC=R{\pi_{Q2R}}_!R{j_2}_!{j_2}^{-1}[R\sigma_!({\pi_{IQ1R1}}^{-1}{\pi_{Q1R1}}^{-1}\GG\otimes{\pi_{IQ12R2}}^{-1}\widehat{\s})].
\end{equation}
Apply commutation relations ${j_2}^{-1}R\sigma_!=R\sigma_! j^{-1}$ and $R{j_2}_!R\sigma_!=R\sigma_!Rj_!$ and the composition $R{\pi_{Q2R}}_!R\sigma_!=R{\pi_\sigma}_!$ to (\ref{aston}) one get
\begin{equation}
\begin{split}
\GG\bu\PC=R{\pi_\sigma}_!Rj_! j^{-1}({\pi_{IQ1R1}}^{-1}{\pi_{Q1R1}}^{-1}\GG\otimes{\pi_{IQ12R2}}^{-1}\widehat{\s})\\
=R{\pi_\sigma}_!(Rj_! j^{-1}{\pi_{IQ1R1}}^{-1}{\pi_{Q1R1}}^{-1}\GG\otimes{\pi_{IQ12R2}}^{-1}\widehat{\s})\\
=R{\pi_\sigma}_!(Rj_!{\pi_{IQ1R1}}^{-1} {j_1}^{-1}{\pi_{Q1R1}}^{-1}\GG\otimes{\pi_{IQ12R2}}^{-1}\widehat{\s}).
\end{split}
\end{equation}
Therefore
\begin{equation}
\begin{split}
Rhom(\GG\bu\PC,\F)
=Rhom(R{\pi_\sigma}_!(Rj_!{\pi_{IQ1R1}}^{-1} {j_1}^{-1}{\pi_{Q1R1}}^{-1}\GG\otimes{\pi_{IQ12R2}}^{-1}\widehat{\s}),\F)\\
\cong Rhom(Rj_!{\pi_{IQ1R1}}^{-1} {j_1}^{-1}{\pi_{Q1R1}}^{-1}\GG\otimes{\pi_{IQ12R2}}^{-1}\widehat{\s},{\pi_\sigma}^!\F)\\
\cong Rhom(Rj_!{\pi_{IQ1R1}}^{-1} {j_1}^{-1}{\pi_{Q1R1}}^{-1}\GG,\underline{Rhom}({\pi_{IQ12R2}}^{-1}\widehat{\s},{\pi_{\sigma}}^!\F))\\
\cong Rhom(\GG,R{\pi_{Q1R1}}_* R{j_1}_* R{\pi_{IQ1R1}}_* j^!\underline{Rhom}({\pi_{IQ12R2}}^{-1}\widehat{\s},{\pi_{\sigma}}^!\F)).
\end{split}
\end{equation}

That is,
\begin{equation}
Rhom(\GG\bu\PC,\F)\cong Rhom(\GG,\EC(\F)).
\end{equation}

\end{proof}

We check that the abovementioned operations are compatible with $\D$.
\begin{lem}\label{Aaron}
For any $\F\in D(Q\x\R)$, one has $\EC(\F)\in\D(Q\x\R)$.
\end{lem}
\begin{proof}

Pick any $\GG\in D_{\leq0}(Q\x\R)$ then in $D(Q\x\R)$ one  has 
\begin{equation}
Rhom(\GG,\EC(\F))\cong Rhom(\GG\bu\PC,\F)\cong Rhom(0,\F)= 0.
\end{equation}
Therefore $\EC(\F)\in {D_{\leq0}}^{\perp,r}$. Apply the fact (\ref{D}) that $\D$ is both left and right orthogonal complement of $D_{\leq0}$.
	
\end{proof}

Now we prove that $\EC$ is microlocally vanishing outside the open sublevel set $B_C$.

\begin{lem}\label{vanishing}
	For all $\F\in\D_{T^*Q\setminus B} (Q\x\R)$, one has $\EC(\F)\cong 0$.
\end{lem}

\begin{proof}

Recall that 
\begin{equation}
\EC(\F):= R{\pi_{Q1R1}}_* R{j_1}_* R{\pi_{IQ1R1}}_* j^! \underline{Rhom}({\pi_{IQ12R2}}^{-1}(\widehat{\s}),{\pi_{\sigma}}^!\F).
\end{equation}

By Lemma \ref{Aaron}, it suffices to show that 
\begin{equation}
R{\pi_{IQ1R1}}_* j^! \underline{Rhom}({\pi_{IQ12R2}}^{-1}(\widehat{\s}),{\pi_{\sigma}}^!\F)\in D_{\leq0}((-\infty,C)\x Q_1\x\R_1).
\end{equation}

By Proposition \ref{pull} one has, up to the zero section $\{\zeta=0\}$, the following estimate:
\begin{equation}
SS({\pi_{\sigma}}^!\F)\subset\{(t,0;q_1,0;q_2,\zeta  p_2;z_1,\zeta;z_2,\zeta)| \zeta>0 ,H(q_2,p_2)\geq C  \} .
\end{equation}

On the other hand, recall from $(\ref{akira})$ one has
\begin{equation}
\begin{split}
SS( {\pi_{IQ12R2}}^{-1}\widehat{\s})\subset \{(t,\zeta_2 s;q_1,\zeta_2 p_1';q_2,\zeta_2 p_2';z_1,0;z_2,\zeta_2)|\\
\zeta_2\geq0, H(q_1,p_1')=H(q_2,p_2')=t, (q_2,p_2')=h_s(q_1,p_1') \}
\end{split}
\end{equation}

It is clear that 
\begin{equation}
SS( {\pi_{IQ12R2}}^{-1}\widehat{\s})\bigcap SS({\pi_{\sigma}}^!\F) \subset T^*_{(I^*\x Q_1\x Q_2\x\R_1\x\R_2)}(I^*\x Q_1\x Q_2\x\R_1\x\R_2).
\end{equation}

Thus by Proposition \ref{tenhom}, one obtain
\begin{equation}\label{austin}
\begin{split}
SS(\underline{Rhom}({\pi_{IQ12R2}}^{-1}\widehat{\s},{\pi_{\sigma}}^!\F))
\subset SS( {\pi_{IQ12R2}}^{-1}\widehat{\s})^a + SS({\pi_{\sigma}}^!\F)\\
\subset \{(t,-\zeta_2 s;q_1,-\zeta_2 p_1';q_2,\zeta p_2-\zeta_2 p_2';z_1,\zeta;z_2,\zeta-\zeta_2)
|\zeta>0,\zeta_2\geq0, \\
H(q_2,p_2)\geq C, H(q_1,p_1')=H(q_2,p_2')=t,(q_2,p_2')=h_s(q_1,p_1')  \}.
\end{split}
\end{equation}

Note that $R{\pi_{IQ1R1}}_* j^! \underline{Rhom}({\pi_{IQ12R2}}^{-1}(\widehat{\s}),{\pi_{\sigma}}^!\F)$ is an object of $D((-\infty,C)\x Q_1\x\R_1)$. We pick an arbitrary vector $v$ in $T^*_{\zeta_1>0}((-\infty,C)\x Q_1\x\R_1)$ that satisfies
\begin{equation}\label{alba}
v=(t,\tau;q_1,p_1;z_1,1)\in SS(R{\pi_{IQ1R1}}_* j^! \underline{Rhom}({\pi_{IQ12R2}}^{-1}(\widehat{\s}),{\pi_{\sigma}}^!\F).
\end{equation}

By Proposition \ref{projection}, there exist sequences $\{-\zeta_2 s\}$,$\{-\zeta_2 p_1'\}$,$\{\zeta p_2 -\zeta_2 p_2'\}$, $\{\zeta\}$, $\{\zeta-\zeta_2\}$ subject to the conditions of (\ref{austin}) such that they and their base coordinates approach to the corresponding entries of $(t,\tau;q_1,p_1;-,0;z_1,1;-,0)$, the extension of $v$ by the zero section, as follows:
\begin{equation}\label{ashley} 
\begin{split}
\{-\zeta_2 s\}&\rightarrow \tau\\
\{-\zeta_2 p_1'\}&\rightarrow p_1\\
\{\zeta p_2 -\zeta_2 p_2'\}&\rightarrow 0\\
\{\zeta\}&\rightarrow 1\\
\{\zeta-\zeta_2\}&\rightarrow 0
\end{split}
\end{equation}

From (\ref{ashley}) we deduce that $\{s\} \rightarrow -\tau$, $ \{p_1'\}\rightarrow -p_1 $ and $\{p_2-p_2'\}\rightarrow 0$. Since the sequences are subject to the condition $(q_2,p_2')=h_s(q_1,p_1')$, we deduce that $(q_2,\lim\{ p_2'\})=h_{-\tau}(q_1,-p_1)$ is finite from the assumption of $H$ generating complete Hamiltonian flow. Therefore $\lim \{p_2\}$ is finite, since $\lim \{p_2-p_2'\}=0$ and $\lim \{p_2'\}$ is finite. The other conditions from (\ref{austin}) then imply that
\begin{equation}\label{apple}
t = H(q_2,\lim\{p_2'\})=H(q_2,\lim \{p_2\})\geq C,
\end{equation}
which contradicts to the assumption
\begin{equation}
(t,\tau;q_1,p_1;z_1,1)=v\in T^*_{\zeta_1>0}((-\infty,C)\x Q_1\x\R_1),
\end{equation} 
that is $t<C$.

\end{proof}
This finishes the proof of Lemma \ref{left}.

\end{proof}
By the general property $SS(\F\oplus\GG)=SS(\F)\cup SS(\GG)$, all the subcategories mentioned above are thick (closed under direct summands). Therefore Lemma \ref{right} and Lemma \ref{left} together imply Theorem \ref{projector}.

\begin{rmk} 
In fact Lemma \ref{right} does not involve the assumption of complete flow. Nevertheless, such assumption is essential in Lemma \ref{left}, where the semiorthogonality needs an additional proof. Note that our argument in proving Lemma \ref{vanishing} is of $C^0$ nature. We use only the continuity of the function $H$ and its flow $h$.
\end{rmk}

\section{A Non-example}

What if the Hamiltonian flow fails to be complete? 

Let us consider the famous Cherry Hamiltonian \cite{Cherry25} on $T^*\R^2$ with parameter $\mu$,
\begin{equation}\label{cherry}
H_{ch}(q,p)=\f{1}{2}({q_1}^2+{p_1}^2)-({q_2}^2+{p_2}^2)+\mu[q_2({q_1}^2-{p_1}^2)-2q_1p_1p_2]
\end{equation}

The Hamiltonian equations of the flow are
\begin{equation}\label{cherryeqn}
\begin{split}
\dot{q_1}&=p_1-2\mu q_2 p_1 -2\mu q_1 p_2 \\
\dot{p_1}&=-q_1 -2\mu q_1q_2+2\mu p_1p_2 \\
\dot{q_2}&=-2p_2 -2\mu q_1p_1 \\
\dot{p_2}&=2q_2 +\mu {p_1}^2-\mu {q_1}^2
\end{split}
\end{equation}

The Cherry Hamiltonian $H_{ch}$ is an example of unstable equilibrium which linearization is stable. In fact, the linear part (corresponding to $\mu=0$) of the Cherry Hamiltonian (\ref{cherry}) is just the superposition of two simple harmonic oscillators. 

For $\mu>0$, on the non-compact level surface $H_{ch}^{-1}(0)$ there is a family of solutions of (\ref{cherryeqn}), parametrized by two parameters $a$ and $b$:
\begin{equation}\label{cherrysol}
\begin{split}
q_1(s)=\f{-\sqrt{2}}{2\mu s + a}\sin(s+b)\quad,&\quad p_1(s)=\f{\sqrt{2}}{2\mu s + a}\cos(s+b)\\
q_2(s)=\f{1}{2\mu s +a}\sin(2s+2b)\quad,& \quad p_2(s)=\f{1}{2\mu s+a}\cos(2s+2b).
\end{split}
\end{equation}

The flow on this $H_{ch}^{-1}(0)$ cannot be made complete. For example when $(a,b)=(0,0)$, the solution (\ref{cherrysol}) is a curve defined for $s\neq0$:
\begin{equation}\label{cherrycurve}
\begin{split}
q_1(s)=\f{-\sin(s)}{\sqrt{2}\mu s}\qquad,&\qquad p_1(s)=\f{\cos(s)}{\sqrt{2}\mu s}\\
q_2(s)=\f{\sin(2s)}{2\mu s}\qquad,& \qquad p_2(s)=\f{\cos(2s)}{2\mu s}
\end{split}
\end{equation}

When $s$ approaches to $\pm \infty$, the solution curve (\ref{cherrysol}) swirls around and approaches to the origin. When $s$ approaches to $0$, its $q$-components become stable but the $p$-components totally blow up, in a way depending on the sign of $s$.

Let us show that $H_{ch}$ violates the projector property by violating the middle equality in (\ref{apple}). For each $n\in\mathbb{N}$, consider two sequences $y_n=(\sqrt{n},\sqrt{n})$ and $y'_n=(\sqrt{n+1},\sqrt{n})$ in $\R^2$ satisfying
\begin{equation}
\begin{split}
\lim \{y_n\}=\lim \{y'_n\}&=(+\infty,+\infty),\\
\lim \{y'_n-y_n\} &= (0,0).
\end{split}
\end{equation}

We choose $\mu>0$ and $x=(q_1,q_2)\in \R^2$ such that $\f{1}{2}+2\mu q_1+\mu q_2=0$. Then from (\ref{cherryeqn}) one has
\begin{equation}
\begin{split}
H_{ch}(x,y_n)&=\f{1}{2}q_1^2-q_2^2+\mu q_1^2q_2-(\f{1}{2}+2\mu q_1+\mu q_2)n\\
& = -2\mu q_1^3 -\f{2}{\mu}q_1 -\f{1}{4\mu^2}
\end{split}
\end{equation}
and
\begin{equation}
\begin{split}
H_{ch}(x,y'_n)&=\f{1}{2}q_1^2-q_2^2+\mu q_1^2q_2+1+(2\mu q_1)[(n+1)-\sqrt{n(n+1)}]\\
& = -2\mu q_1^3 -\f{2}{\mu}q_1 -\f{1}{4\mu^2} +1 + (2\mu q_1)[(n+1)-\sqrt{n(n+1)}]
\end{split}
\end{equation}

It follows that when $n\rightarrow \infty$, one has 
\begin{equation}
\lim H_{ch}(x,y_n) = -2\mu q_1^3 -\f{2}{\mu}q_1 -\f{1}{4\mu^2},
\end{equation}
while on the other hand
\begin{equation}
\lim H_{ch}(x,y'_n) = -2\mu q_1^3 -\f{2}{\mu}q_1 -\f{1}{4\mu^2} +1 + \mu q_1.
\end{equation}

The term $H_{ch}(x,y_n)$ is independent of $n\in\mathbb{N}$ and takes every real value as $x$ varies, and the term $\lim H_{ch}(x,y'_n)$ is always different from $\lim H_{ch}(x,y_n)$ unless $x=(\f{-1}{\mu},\f{3}{2\mu})$. This allows us to choose a sheaf $\F\in\D_{T^*\R^2 \setminus B_C}(\R^2\x\R)$ supported on $\{\f{1}{2}+2\mu q_1+\mu q_2=0\}$, and a vector 
\begin{equation}
v\in T^*_{\zeta_1>0} \cap SS(R{\pi_{IQ1R1}}_* j^! \underline{Rhom}({\pi_{IQ12R2}}^{-1}(\widehat{\s}),{\pi_{\sigma}}^!\F)
\end{equation}
as in (\ref{alba}), which causes $\EC(\F)\neq0$.

\bibliographystyle{amsplain}

\begin{thebibliography}{100}


	\bibitem{Cherry25} T. M. Cherry,
	\emph{Some examples of trajectories defined by differential equation of a generalized dynamical type}, Trans. Cambridge Phil. Soc. 23 (1925) 169-200.
	
	\bibitem{Chiu17} S.-F. Chiu, \emph{Nonsqueezing property of contact balls}, Duke Math. J. Vol 166 (2017), No. 4, pp. 605-655 .
	
	\bibitem{GKS12}
	S. Guillermou, M. Kashiwara and P. Schapira,
	\emph{Sheaf quantization of Hamiltonian isotopies and applications to non displaceability problems.}
	Duke Math. J. Vol 161 (2012), pp.201--245.
		
	\bibitem{KO21}
	M. Kawasaki and R. Orita,
	\emph{Rigid fibers of integrable systems on cotangent bundles},
	 J. Math. Soc. Japan Advance Publication 1-19, October, 2021. https://doi.org/10.2969/jmsj/84278427
		
	\bibitem{KS90}
	M. Kashiwara and P. Schapira,
	\emph{Sheaves on Manifolds}, 
	Grundlehren der Math. Wiss. \textbf{292} Springer-Verlag (1990).

	\bibitem{MVZ12}
	A. Monzner, N. Vichery, and F. Zapolsky,
	\emph{Partial quasimorphisms and quasistates on cotangent bundles, and symplectic homogenization},
	Journal of Modern Dynamics,	2012, 6(2): 205-249.
	doi: 10.3934/jmd.2012.6.205 
	
	\bibitem{Tamarkin18} D. Tamarkin, \emph{Microlocal conditions for non-displaceability}, in Algebraic and analytic microlocal analysis, edited by M. Hitrik et al., Springer Proceedings	in Mathematics and Statistics 269 (2018), pp. 99-223.

  
\end{thebibliography}

\end{document}